\documentclass[12pt, reqno]{amsart}
\usepackage{amsmath,amssymb,amsthm,graphicx,mathrsfs}
\usepackage{amsmath,amssymb,amsthm,tikz}
\usepackage{tikz-cd}
\usepackage[hidelinks]{hyperref}

\usetikzlibrary{decorations.pathreplacing}
\usepackage[toc,page]{appendix}

\usepackage[margin=1in]{geometry}
\usepackage{comment}
\newcommand{\R}{\mathbb R}

\newtheorem{theorem}{Theorem}[section]
\newtheorem{lemma}[theorem]{Lemma}
\newtheorem{proposition}[theorem]{Proposition}

\newtheorem{introthm}{Theorem}

\newtheorem{introcorollary}[introthm]{Corollary}

\theoremstyle{remark}
\newtheorem{remark}[theorem]{Remark}

\newcommand{\eps}{\varepsilon}

\numberwithin{equation}{section}

\title{Monotonicity of the Liouville entropy\\ along the Ricci flow on surfaces}

\author{Karen Butt, Alena Erchenko, Tristan Humbert, and Daniel Mitsutani}

\date{}

\begin{document}

\begin{abstract} 
We show that the Liouville entropy of the geodesic flow of a closed surface of non-constant negative curvature is eventually strictly increasing along the normalized Ricci flow (NRF). More precisely, we obtain a new expression for the derivative of the Liouville entropy along an arbitrary conformal deformation in dimension 2, and we prove it is positive in the direction of the NRF for 1/6-pinched metrics. This partially answers a question of Manning from 2004. 
In addition, we show that the mean root curvature, a purely geometric quantity which is a lower bound for the Liouville entropy, is strictly increasing along the NRF starting from any metric of non-constant negative curvature.
\end{abstract}

\maketitle

\section{Introduction}

Let $(M, g)$ be a closed negatively curved surface, and let $h_{\rm Liou}(g)$ denote its \emph{Liouville entropy}, i.e., the measure-theoretic entropy of the geodesic flow on its unit tangent bundle 
$S^g M$ 
with respect to the Liouville measure.
In this paper, we partially answer
a question raised by Manning \cite[Question 3]{Man_Ricci} about
the monotonicity of $h_{\rm Liou}(g)$ along the normalized Ricci flow on the space of negatively curved metrics on $M$. Recall that for $\alpha>0$, a negatively curved metric is $\alpha$-\emph{pinched} if there is $C>0$ such that $-C\leq K_g\leq -\alpha C$ where $K_g$ denotes the Gaussian curvature of $g$. 

\begin{introthm}
\label{theo:Liou}
Let $M$ be a smooth closed orientable surface of negative Euler characteristic. 
Let $g_0$ be a smooth Riemannian metric on $M$ of non-constant negative Gaussian curvature, and suppose that $g_0$ is $1/6$-pinched.
Let $\eps \mapsto g_{\eps}$ denote the normalized Ricci flow starting from $g_0$.  Then $$\eps \mapsto h_{\rm Liou}(g_{\eps}) \text{ is strictly increasing for all }\eps \geq 0.$$ 
\end{introthm}
We recall that in dimension 2, the normalized Ricci flow (NRF) is given by
\begin{equation}\label{eq:RicciFlow} 
\frac{\partial}{\partial \eps}g_\eps = -2(K_\eps-\bar{K})g_\eps,
\end{equation}
where
$K_\eps$ is the Gaussian curvature of $g_\eps$ and $\bar {K}$ is its average value, which is independent of $\eps$ by Gauss--Bonnet.
Hyperbolic metrics, i.e., metrics of constant Gaussian curvature, are fixed by the Ricci flow; for metrics of non-constant curvature, \eqref{eq:RicciFlow} defines a conformal family of negatively curved metrics $\eps \mapsto g_{\eps}$ of fixed area converging exponentially fast (in the $C^k$ norm for any $k\geq 0$) to a hyperbolic metric (of constant curvature $\bar{K}$) as $\eps \to \infty$ \cite[Theorem 3.3]{hamilton1988ricci}. 
 In particular, given any 
 metric $g_0$, the curvatures $K_{\eps}$ converge uniformly to $\bar{K}$ when $\eps\to \infty$, so that $g_\eps$ is eventually $1/6$-pinched. Hence, Theorem \ref{theo:Liou} implies
\begin{introcorollary}
\label{corr:corr}
   Let $M$ be a smooth closed orientable surface of negative Euler characteristic. 
Let $g_0$ be a smooth Riemannian metric on $M$ of non-constant negative Gaussian curvature. Then there is a $T>0$, such that  
$$\eps \mapsto h_{\rm Liou}(g_{\eps}) \text{ is strictly increasing for all }\eps \geq T.$$
\end{introcorollary}
\begin{remark}
Since $g_0$ is a negatively curved metric on a closed surface, there exists $\alpha$ such that $g_0$ is $\alpha$-pinched. 
    Let $(g_\eps)_{\eps\geq 0}$ be the NRF starting from $g_0.$ For $\eps\geq 0$, let $K_{\mathrm{min}}(\eps)$ (resp. $K_{\mathrm{max}}(\eps)$) denote the minimum (resp. maximum) of $K_\eps$. By the proof of \cite[Theorem 3.3]{hamilton1988ricci}, we have $$\tfrac{d}{d\eps}(K_{\mathrm{max}}(\eps)-K_{\mathrm{min}}(\eps))\leq K_{\mathrm{max}}(0)(K_{\mathrm{max}}(\eps)-K_{\mathrm{min}}(\eps)).$$ 
As a consequence, $g_\eps$ is $1/6$-pinched for all $\eps\geq T$ where $T=\ln\left(5\alpha/6\right)/K_{\mathrm{max}}(0)$.
\end{remark}

Manning considered the variation of the \emph{topological entropy} along the normalized Ricci flow, proving it decreases starting from any negatively curved metric \cite{Man_Ricci}.
The topological entropy is an invariant which describes the total exponential complexity of a dynamical system, whereas the Liouville entropy encodes its average complexity with respect to the Liouville (volume) measure.
Katok proved that Liouville entropy (resp. topological entropy) is maximized (resp. minimized) at hyperbolic metrics among negatively curved metrics of the same area \cite[Theorem B]{Ka}. Moreover, he showed
these two entropies coincide if and only if the metric $g$ is hyperbolic \cite[Corollary 2.5]{Ka}.

 We note the latter result follows for $1/6$-pinched metrics from Theorem \ref{theo:Liou} and \cite[Theorem 1]{Man_Ricci}. 
Indeed, if $g$ is $1/6$-pinched but not hyperbolic, the difference $h_{\rm top}(g) - h_{\rm Liou}(g)$ is \emph{strictly} decreasing along the Ricci flow. On the other hand, the variational principle states $h_{\rm top}(g) - h_{\rm Liou}(g) \geq 0$, so the inequality must be strict.

\subsection*{Mean root curvature}
Our next result concerns a geometric invariant introduced by Manning \cite{Man2} known as the \textit{mean root curvature}, which is defined for a negatively curved metric $g$ on a closed surface $M$ by
\begin{equation}
\label{eq:mrc_def}
\kappa(g): =\frac{1}{A(g)}\int_{M}\sqrt{-K_g}\, dA_g,
\end{equation}
where $dA_g$ is the Riemannian area form of $g$ and $A(g)=\int_M dA_g$ is the total area. 

The mean root curvature is small for metrics which concentrate curvature in regions of small area, and is maximized strictly at metrics of constant negative curvature by Jensen's inequality and the Gauss--Bonnet theorem.  In addition, by \cite[Theorem 2]{Man2} and \cite[Corollary 1]{Sar}, it provides a lower bound for the Liouville entropy: $\kappa(g) \leq h_{\rm Liou}(g)$. Moreover, the equality holds if and only if $g$ has constant negative Gaussian curvature \cite{OS84}. 

Since the mean root curvature is a purely geometric invariant related to the concentration of Gaussian curvature and to the Liouville entropy, it is natural to ask if it is also strictly increasing along the Ricci flow. We prove that this is indeed the case.

\begin{introthm}
\label{theo:meancurv}
Let $M$ be a smooth closed orientable surface of negative Euler characteristic. 
Let $g_0$ be a smooth Riemannian metric on $M$ of non-constant negative Gaussian curvature.
Let $\eps \mapsto g_{\eps}$ denote the normalized Ricci flow starting from $g_0$. 
Then $$\eps \mapsto \kappa(g_{\eps}) \text{ is strictly increasing for all }\eps \geq 0.$$ 
\end{introthm}

\subsection*{Strategy of the proofs}

 In Section \ref{section:meancurv}, we prove  Theorem \ref{theo:meancurv} by first finding the derivative of $\kappa(g)$ along an arbitrary conformal perturbation (Proposition \ref{meancurv-conf}). We then deduce positivity of this derivative along the Ricci flow using a Jensen-type inequality (Lemma  \ref{lemma: integral inequality}).

In Section~\ref{sec:Liou}, we prove Theorem \ref{theo:Liou}. The key ingredient is a new formula for the derivative of the Liouville entropy along an arbitrary area-preserving conformal  perturbation of a negatively curved metric on a surface. The formula involves integrals of functions over forward infinite orbits of the geodesic flow, which we call \emph{half-orbit integrals}; see \eqref{eq:defIf} for the precise definition.

\begin{introthm}
\label{thm:conf-deriv}
Let $(M,g_0)$ be a smooth closed negatively curved surface. Denote by $m$ the Liouville measure on its unit tangent bundle $SM$. Let $g_{\epsilon}=e^{2\rho_{\epsilon}}g_0$ be a $C^{\infty}$ area-preserving conformal perturbation of $g_0$ and write  $\dot{\rho}_0=\tfrac{d}{d\epsilon}|_{\epsilon=0}\, \rho_{\epsilon}$. 
Then
\begin{equation}
\label{eq:formula}
    \left. \frac{d}{d \eps} \right|_{\eps = 0} h_{\rm Liou}(g_\eps)=  - \int_{SM}\dot{\rho}_0 w^sdm  - 2 \int_{SM} V(w^s) I_{\dot \rho_0}dm,
\end{equation}
where $- w^s(v)$ is the geodesic curvature of the stable horocycle determined by $v$ $($see \eqref{eq:2ndff}$)$, $V$ is the vertical vector field $($see \eqref{eq:V}$)$, and $I_{\dot{\rho}_0}$ is the half-orbit integral defined by $\dot{\rho}_0$ $($see \eqref{eq:defIf}$)$. 
\end{introthm}

To prove Theorem \ref{thm:conf-deriv}, we begin with the well-known fact that, in negative curvature, the Liouville entropy can be expressed as the average, with respect to the Liouville measure, of the mean curvature of horospheres (see (\ref{eq:hDivX})). 
This was used by Knieper--Weiss to show the Liouville entropy varies smoothly with respect to the metric for negatively curved surfaces \cite{KniWei}. 

In this paper, we use that the mean curvature of a horosphere is in turn equal to the Laplacian of the corresponding Busemann function, and can hence be expressed as the divergence of a vector field closely related to the geodesic spray. 
This formulation of the mean curvature was used by Ledrappier--Shu in \cite{LedShu_conformal, LedShu_general} to study the differentiability of the linear drift. In Section \ref{sub:conf}, we differentiate the horospherical mean curvature using their methods.
A key tool, in both their work and ours, is a slightly non-standard decomposition of the unit tangent bundle of the universal cover $\tilde M$ as the product of $\tilde M$ with $\partial \tilde M$, the visual boundary at infinity.
As a consequence of this perspective, integrals of certain functions along half-infinite orbits of the geodesic flow appear naturally in the computations.  

As shown in Section \ref{sec:half}, these half-orbit integrals satisfy a certain differential equation, which allows us to to simplify our derivative formula enough to prove Theorem \ref{theo:Liou} for metrics with 1/6-pinched sectional curvature in Section \ref{sub:specRicci}.
More precisely, as in the proof of Theorem C, we deduce positivity of the first term in \eqref{eq:formula} using Lemma \ref{lemma: integral inequality}, whereas the hypothesis on the pinching is used only at the very end to show the positivity of the second term. 

\subsection*{Acknowledgments}
This project began while the authors attended an American Institute of Mathematics (AIM) SQuaRE workshop in September 2024. We would like to thank AIM for this opportunity and for their hospitality during our stay. 
We thank Fran\c{c}ois Ledrappier for answering our questions about the works \cite{LedShu_conformal, LedShu_general}.
We also thank Colin Guillarmou, Thibault Lefeuvre and
Amie Wilkinson for comments on an earlier draft of this paper.

The first author was supported by NSF grant DMS-2402173. The second author was supported by NSF grant DMS-2247230.
The third author was supported by the European Research Council (ERC) under
the European Union’s Horizon 2020 research and innovation programme (Grant agreement no. 101162990 — ADG).

\section{Preliminaries}
In Section \ref{secSM} we record standard facts on the geometry of the unit tangent bundle of a surface, and in Section \ref{secAnosov} we describe the stable and unstable distributions of the geodesic flow in negative curvature. In Section \ref{secDiv}, we recall that in our setting, the Liouville entropy has a geometric formulation as the average of the mean curvatures of horospheres (see \eqref{eq:hDivX}), which is the starting point of our proof of Theorem \ref{thm:conf-deriv}. 
In Section \ref{secInt}, we record a Jensen-type integral inequality that is used in the proofs of Theorems \ref{theo:Liou} and \ref{theo:meancurv}.

\subsection{Geometry of surfaces}
\label{secSM}
In this section, we recall some basic facts about the geometry of  surfaces and establish some notation. For a textbook account of all these notions, we refer to \cite[Chapter 1]{Pat} and \cite[Chapter 2]{DGRbook}. 

\subsubsection{Geodesic flow and Liouville measure}\label{subsub:Liou-meas}
Consider a smooth closed surface $M$ equipped with a smooth Riemannian metric $g$. Let $K_g$ denote the sectional curvature of $g$. We will denote by $dA_g$ the Riemannian area form defined by $g$ on $M$. Let $A(g) = \int_MdA_g$ be the total area. 

Let $SM = \{(x,v)\in TM\mid \|v\|_g=1\}$ denote the unit tangent bundle of $g$.  
The pair $(M,g)$ defines a dynamical system on $SM$ called the \emph{geodesic flow}:
\begin{equation}
\label{eq:phit}
\varphi_t:(x, v) \mapsto (\gamma_v(t), \dot \gamma_v(t)),
\end{equation}
where $t\mapsto \gamma_v(t)$ is the (projection on $M$) of the unique geodesic passing through $x$ at time $t=0$ with velocity $v$. 
The \emph{geodesic spray} is the vector field generating $\varphi_t$, i.e., 
\begin{equation}
    \label{eq:X}
    X(x,v):=\frac{d}{dt}\Big|_{t=0}\varphi_t(x,v)\in C^{\infty}(SM,T(SM)).
\end{equation}

The metric $g$ induces a probability measure on $SM$ called the \emph{Liouville measure}, which we will denote by $m = m_g$.  
This measure has a concrete description which is compatible with the sphere-bundle structure of $SM$: it is locally given (up to a multiplicative constant) by the product of the Riemannian area $dA_g$ on the base $M$, together with the spherical Lebesgue measure (arclength) on the circular fibers. 
This measure also turns out to be geodesic-flow--invariant, as we will discuss below. In summary, the metric $g$ defines a measure-preserving dynamical system $(SM, \varphi_t, m)_g$.

\subsubsection{Horizontal and vertical vector fields}
We now recall a geometric splitting of the unit tangent bundle $T(SM)$. 
We define a vertical vector field as follows.
An oriented Riemannian surface admits a complex structure. 
This means that there is a section $J\in \mathrm{End}(TM)$ satisfying $J^2=-\mathrm{Id}$, and such that the area form associated to $g$ is given by $dA_g=g(J\cdot,\cdot)$. 
One defines a rotation in the fiber by
$$\rho_{\theta}:SM\to SM,\quad \rho_\theta(x,v)=(x,e^{\theta J}v),   $$
where $e^{\theta J}v$ is the unit vector obtained by rotating $v$ by an angle $\theta$ in the positive direction (with respect to the orientation of $M$).
The \emph{vertical vector field} $V$ is the generator of $\rho_\theta$:
\begin{equation}
\label{eq:V}
V:=\left.\frac{d}{d\theta}\right|_{\theta=0}\rho_\theta\in C^{\infty}(SM,T(SM)).
\end{equation}
By the definition of $\rho_\theta$, we see that $V(x,v)$ is the vertical lift of $Jv\in T_xM.$
Next, we define the horizontal vector field $H:=[V,X]$. We will use the following important commutation relations; see for instance \cite[Lemma 15.2.1]{Lef},
\begin{equation}
\label{eq:comm}
H=[V,X], \quad [H,V]=X, \quad [X,H]=K_gV.
\end{equation}
One defines the \emph{Sasaki metric} $g_{\mathrm{Sas}}$ on $T(SM)$ by declaring $(X,V,H)$ to be a global orthonormal frame. This metric defines a (normalized) Riemmannian volume form on $SM$ which coincides with the \emph{Liouville measure} $m_g$ defined above, see \cite[Lemma 1.30]{GuMaz}.
One can also show (see \cite[Exercise 1.33]{Pat}) that there is a contact structure on $SM$ for which $X$ is the Reeb vector field and $m_g$ is the Liouville form. In particular, we deduce the important property that the Liouville measure is $\varphi_t$-invariant. Moreover, the Liouville measure is invariant with respect to $H$ and $V$, see \cite[Proposition 1.47]{GuMaz}. In other words,
\begin{equation}
    \label{eq:Adjoints}
    X^*=-X, \quad H^*=-H, \quad V^*=-V,
\end{equation}
where $Y^*$ denotes the $L^2(SM,dm)$-adjoint of a differential operator $Y$.

\subsection{The Anosov property and (un)-stable manifolds}
\label{secAnosov}
The main hypothesis in this paper is that the curvature of $g$ is negative, that is $K_g<0$. This ensures that the dynamics of the geodesic flow are \emph{chaotic}.

\begin{proposition}[\cite{Anosov}]
\label{PropAnosov}
The geodesic flow on a negatively curved manifold $(M,g)$ is Anosov $($uniformly hyperbolic$)$. That is, there exist constants $C,\lambda>0$, together with a flow-invariant and continuous splitting 
\begin{equation}
\label{eq:split}T (SM)=E^s\oplus \mathbb R X\oplus E^u,
\end{equation}such that
\begin{equation}
\label{eq:Anosov}
\forall v\in SM,\quad  \begin{cases}\|d\varphi_t (v)W^s\|_{g_{\mathrm{Sas}}}&\leq Ce^{-\lambda t}\|W^s\|_{g_{\mathrm{Sas}}}, \ W^s\in E^s(v), \ t\geq 0,
\\\|d\varphi_t (v)W^u\|_{g_{\mathrm{Sas}}}&\leq Ce^{-\lambda |t|}\|W^u\|_{g_{\mathrm{Sas}}}, \ W^u\in E^u(v), \ t\leq 0.
\end{cases}
\end{equation} 
The bundle $E^s$ $($resp. $E^u)$ is called the stable $($resp. unstable$)$ bundle of the flow.
\end{proposition}
See, for instance, \cite{Bal} for a proof of this proposition.

\subsubsection{Stable manifolds and horocycles}
\label{buse}We start by describing the stable and unstable \emph{manifolds} of the flow.
For any $v \in  SM$, these are, by definition, immersed submanifolds 
\begin{equation}
\label{eq:W^{s,u}}
\begin{split}
\mathcal W^{s}(v)&:=\{v'\in S M\mid \lim\limits_{t\to +\infty}d(\varphi_t(v),\varphi_t(v'))=0\}, 
\\
\mathcal W^{u}(v)&:=\{v'\in S M\mid \lim\limits_{t\to -\infty}d(\varphi_t(v),\varphi_t(v'))=0\},
\end{split}
\end{equation}
 called the (strong) stable (resp. unstable) manifolds, such that 
$ T_v \mathcal W^{s}=E^{s}(v)$ and $ T_v \mathcal W^{u}=E^{u}(v).$  
We also define the \emph{weak} stable and unstable manifolds
\begin{equation}
\label{eq:W^{ws,wu}}
\begin{split}
\mathcal W^{cs}(v)&:=\{v'\in S M\mid \limsup_{t\to+\infty} \ d(\varphi_{t}(v),\varphi_{t}(v'))<+\infty\}=\bigcup_{t\in \mathbb R}\varphi_t(\mathcal W^{s}(v)),
\\
\mathcal W^{cu}(v)&:=\{v'\in S M\mid \limsup_{t\to -\infty} \ d(\varphi_{t}(v),\varphi_{t}(v'))<+\infty\}=\bigcup_{t\in \mathbb R}\varphi_t(\mathcal W^{u}(v)). 
\end{split}
\end{equation}
Their tangent spaces are given respectively by $\mathbb R X\oplus E^s$ and $\mathbb R X\oplus E^u$.

Geometrically, we can describe the strong/weak stable/unstable manifolds  in terms of \emph{Busemann functions}. 
To lighten the presentation, we will describe the stable case only. 
Let $\tilde M$ denote the universal cover of $M$ and let $\partial \tilde M$ denote its visual boundary at infinity; see for instance, \cite[Chapter 8]{BH}, \cite[Chapter II]{Bal}. 
Let $\pi: S \tilde M \to \partial \tilde M$ denote the natural forward projection along the geodesic flow. 
We have the identification
\begin{equation}\label{eq:Pi}
    \Pi: S \tilde M \to \tilde M \times \partial \tilde M, \quad (x, v) \mapsto (x, \pi(x,v)).
\end{equation}
For $(x, \xi) \in SM$, let $b_{x, \xi} \in C^\infty(\tilde{M})$ denote the associated \emph{Busemann function}: 
\begin{equation}\label{eq:Busemann}
    b_{x,\xi}(p) = \lim_{t \to \infty} (d(p, \gamma_{v}(t)) - t),
\end{equation}
where $\gamma_{v}$ is the geodesic such that $\gamma_{v}(0)=x$ and $\pi(x,v) = \xi$
(see, for instance, \cite[Chapter II]{Bal}). 
For any fixed $\xi \in \partial \tilde M$, the dependence of $b_{x, \xi}(p)$ on $p$ and also on $x$ is $C^\infty$ (see e.g. \cite[Proposition 2.2]{wilkinsonmls}), whereas the dependence on $\xi$ is in general only H{\"o}lder continuous, even though $g$ is a smooth metric. 
Nevertheless, when $\dim M = 2$, it follows from the work of Hurder--Katok \cite{HuKa} that the dependence in $\xi$ is $C^{1+\alpha}$ for some $\alpha>0.$ 
Level sets of Busemann functions are called \emph{horospheres}, or \emph{horocycles} in the case where $\dim M = 2$. 

Fix $\xi \in \partial \tilde M$ and define the vector field $X^{\xi}(y) = -{\rm grad}\,  b_{x,\xi}(y)$ for $y \in \tilde M$. 
Then for $v = (x, \xi) \in SM$, the lift of $\mathcal{W}^{s}(v)$ to $S \tilde M$ is given by (see \cite[p. 72]{Bal}) 
\begin{equation}\label{eq:liftofWs}
   \widetilde{\mathcal{W}^s(v)} =  \{ X^{\xi} (y) \, | \, y \in \{ b_{x, \xi} = 0 \} \}.
\end{equation}
Similarly, the lift of $\mathcal{W}^{cs}(v)$ to $S \tilde M$ is given by
\begin{equation}\label{eq:liftofWcs}
   \widetilde{\mathcal{W}^{cs}(v)} =  \{ X^{\xi} (y) \, | \, y \in \tilde M \}.
\end{equation}

A Jacobi field associated to the geodesic variation in \eqref{eq:liftofWs} is called a \emph{stable Jacobi field}. Since such a Jacobi field is everywhere perpendicular to the geodesic $\gamma_v$ determined by $v\in SM$, and $\dim M = 2$, we can view it as a real-valued function along the geodesic $\gamma_v(t)$. Letting $j^s(t)$ denote this function, we have $j^s(t) \to 0$ as $t \to \infty$. We will call $j^s$ a \textit{stable Jacobi function}.
The exponential decay estimates in the Anosov property (\ref{eq:Anosov}) are equivalent to analogous decay estimates for $j^s(t)$ and $Xj^s(t)$. 
\subsubsection{The stable vector field}
\label{sec:w^s}
We now specify a vector field $e^{s}$ which spans the stable bundle $E^{s}$. 
Let $s \mapsto c(s)$ be a parametrization of the horocycle $\{ b_{x, \xi} = 0 \}$ such that $c(0) = x$ and $c'(0) = J (x, \xi)$, where $J$ is the complex structure of $M$ discussed in the previous section. 
We define the \emph{stable vector field} $e^s$ on $S\tilde M$ by $\left(\Pi_{*}e^s\right)(x,\xi) = \frac{d}{ds}|_{s=0} (c(s), \xi)$. By construction, $e^s$ has integral curves given by $\mathcal{W}^s$. Moreover, since the horizontal component of $e^s$ is $Jv$, we see that $e^s$ is of the form
\begin{equation}\label{eq:e^s}
   e^s  = H + w^s V  
\end{equation}
for some function $w^s:S \tilde M \to \R$.

By the above discussion, the regularity of $w^s$ is $C^{1+\alpha}$ in the setting $\dim M = 2$ \cite{HuKa}.  
We have the following two characterizations of $w^s$, both of which are used crucially in this paper:
\begin{itemize}
   \item $-w^s(v)$ is the (trace of the) second fundamental form, i.e., the mean curvature, of the horosphere $\{ b_v = 0\}$ (or, since $\dim M = 2$, the geodesic curvature of the horocycle). Since the trace of the second fundamental form of a level hypersurface is given by the Laplacian of its defining function, we obtain
    \begin{equation}\label{eq:2ndff}
        -w^s(\Pi^{-1}(x,\xi)) = \Delta b_{x,\xi}(x) = -{\rm Div} (X^{\xi})(x).
    \end{equation}
    \item $w^s = \frac{X j^s}{j^s}$, where $j^s$ is the stable Jacobi function along $\gamma_v$ defined above. In particular, $w^s$ is everywhere negative. Moreover, since $j^s$ satisfies the Jacobi equation, a direct computation shows that $w^s$ satisfies the \emph{Riccati equation}
    \begin{equation}\label{eq:Riccati}
        X(w^s) = -(w^s)^2 - K.
    \end{equation}
\end{itemize}
%
    Note that since $w^s=X(j^s)/j^s = X(\ln(j^s))$, one has 
    \begin{equation} 
    \label{eq:contract}
   \frac{\|d\varphi_t(v)e^s( v)\|}{\|e^s(v)\|}=\frac{j^s(\varphi_t(v))}{j^s(v)}=\exp\left(\int_0^tX(\ln(j^s))(\varphi_rv)dr\right)=e^{\int_0^tw^s(\varphi_r p)dr}.
    \end{equation}
    Since $\sup\limits_{v\in SM}w^s(v)<0$, this shows that $e^s$ is indeed exponentially contracted along the flow, which is consistent with the fact that the stable foliation $E^s$ is tangent to $\mathcal W^s.$

\subsection{Liouville entropy}
\label{secDiv}
The main object of study of this paper is the measure-theoretic entropy of the geodesic flow with respect to the Liouville measure $m_g$, which we denote by $h_{\mathrm{Liou}}(g)$ from now on.  
This invariant roughly captures the exponential rate of divergence of nearby geodesics for $m_g$-a.e. point; 
 see, for instance, \cite[Proposition 1.6]{Ka} or \cite[Appendix A]{FishHas}. 
We recall that in our setting
\begin{itemize}
    \item the geodesic flow $(\varphi_t)_{t\in \mathbb R}$ is Anosov, see Proposition \ref{PropAnosov};
    \item the Liouville measure $m_g$ is \emph{smooth}, meaning, it is equal to the normalized Riemannian volume on $SM$ induced by the Sasaki metric.
\end{itemize}
We can thus use the theory of \emph{thermodynamic formalism} to write $h_{\mathrm{Liou}}$ in terms of the \emph{stable Jacobian}. 
The stable Jacobian of a general Anosov flow is given by the following formula 
\begin{equation*}
\label{eq:Ju1}
J^s(v):=-\left.\frac{d}{dt}\mathrm{det}(d\varphi_t(v)_{|E^s(v)})\right|_{t=0}=-\left.\frac{d}{dt}\ln\mathrm{det}(d\varphi_t(v)_{|E^s(v)})\right|_{t=0}.
\end{equation*}

It is well known that the thermodynamic equilibirium measure associated to $-J^s$ is the Liouville measure; 
see for instance \cite[Theorem 7.4.14]{FishHas}. 
Using 
(\ref{eq:contract}), we have 
$J^s(v)=-w^s(v),$
which implies, by \cite[Corollary 7.4.5]{FishHas},
that we have
\begin{equation}
\label{eq:Liouvilledef}h_{\mathrm{Liou}}(g)=\int_{SM}J^s_g(v)dm_g(v)=-\int_{SM}w^s_g(v)dm_g(v).
\end{equation}
This formula will be the starting point for our proof of Theorem \ref{thm:conf-deriv}.

\begin{remark}
  Alternatively, one can deduce (\ref{eq:hDivX}) using \emph{Lyapunov exponents}, via \emph{Pesin's entropy formula} \cite{Pesin_entropy}. See, for instance, \cite[p. 354]{Man2} and \cite[Appendix A]{KniWei} for accounts of this approach. 
\end{remark}

\begin{remark}\label{rem:MRC}
The mean root curvature (see \eqref{eq:mrc_def}) is conceptually related to the Liouville entropy as follows: averaging both sides of the Riccati equation \eqref{eq:Riccati} with respect to Liouville measure shows that the average of $(w^s)^2$ coincides with that of $-K$; thus one might expect the Liouville entropy, which is the average of $-w^s$, to be related to the average of $\sqrt{-K}$. Indeed, as mentioned in the introduction, Manning proved the former is always larger than the latter \cite[Theorem 2]{Man2}.
\end{remark}


\subsection{A Jensen-type inequality}

\label{secInt}
To show positivity of the derivatives of both the Liouville entropy and mean root curvature, we will need the following lemma. 

\begin{lemma}\label{lemma: integral inequality}
Let $(\Omega,\mu)$ be a probability space. Let $F\colon \Omega\rightarrow \mathbb R$ be a measurable and non-negative function. Then,
$$\int_\Omega F^2\left(F-\int_\Omega F\,d\mu\right)\,d\mu\geq 0,$$
with equality if and only if $F$ is $\mu$-a.e constant.
\end{lemma}

\begin{proof}
We denote $c=\int_\Omega F\,d\mu\geq 0$. Let $\Omega_c=\{x\in \Omega\,|\, F(x)\leq c\}$. Note that if $x\in \Omega_c$ then $F^2(x)\leq c^2$ and $F(x)-c\leq 0$, so $F^2(x)(F(x)-c)\geq c^2(F(x)-c)$. Similarly, if $x\in \Omega\setminus \Omega_c$, then $F^2(x)(F(x)-c)\geq c^2(F(x)-c)$. Thus,
\begin{align*}
\int_\Omega F^2\left(F-\int_\Omega F\,d\mu\right)\,d\mu
&=\int_{\Omega_c}F^2\left(F-c\right)\,d\mu+\int_{\Omega\setminus \Omega_c}F^2\left(F-c\right)\,d\mu\\
&\geq c^2\left(\int_{\Omega_c}\left(F-c\right)\,d\mu+\int_{\Omega\setminus \Omega_c}\left(F-c\right)\,d\mu\right)=c^2\int_{\Omega}(F-c)\,d\mu.
\end{align*}
Since $\mu$ is normalized, we have $\int_\Omega c \, d\mu = c = \int_\Omega F \, d\mu$, which shows the last line above equals 0. The equality holds if and only if $F$ is $\mu$-a.e constant to $c$.
\end{proof}

\section{Monotonicity of the mean root curvature}\label{section:meancurv}

In this section, we prove the mean root curvature $\kappa$ defined in \eqref{eq:mrc_def} is monotonically increasing along the normalized Ricci flow (Theorem \ref{theo:meancurv}). 
First, we compute the variation of the mean root curvature with respect to an area-preserving conformal change. Since we are only interested in the sign of the derivative, we can suppose without loss of generality that $A(g_{\eps})\equiv 1.$ We will use $\epsilon$ as a subscript to indicate that the corresponding objects are defined with respect to the metric $g_\epsilon.$ We consider a smooth one-parameter family of conformal area-preserving changes of $g_0$ :
\begin{equation}\label{eq:conformal}
    g_{\eps} = e^{2 \rho_{\eps}} g_0,\qquad A(g_{\epsilon})=\int_M e^{2\rho_\epsilon(x)}dA_0(x)\equiv A(g_0) .
\end{equation}
We let $\dot \rho_0 \in C^{\infty}(M)$ denote the variation of the conformal factor $\frac{d}{d \eps}|_{\eps = 0} \rho_{\eps}$. 

\begin{proposition}\label{meancurv-conf}
    Let $(M, g_0)$ be a closed surface of 
    negative curvature and area $1$. Let $\eps \mapsto g_{\eps} = e^{2 \rho_{\eps}} g_0$ be a conformal area-preserving deformation of $g_0$.
    Then we have
    \begin{equation}\label{eq:meancurv-conf}
    \dot \kappa_0:=\left.\frac{d} {d\epsilon}\right|_{\epsilon=0}\kappa(g_\epsilon) = \int_M \frac{\Delta_0 \dot \rho_0}{2 \sqrt{-K_0}} dA_0 + \int_M \dot\rho_0 \sqrt{-K_0}\, dA_0.    
    \end{equation}
\end{proposition}

\begin{proof} 
Using \eqref{eq:mrc_def} and the Leibniz rule, we have 
\begin{equation}
    \dot \kappa_0 = \int\limits_M \left. \frac{d}{d\eps} \right|_{\eps=0} \sqrt{-K_\eps} \,dA_0+\int\limits_M 
    \sqrt{-K_0}\left.\frac{d}{d\eps}\right|_{\eps=0}\left(dA_\eps \right)
\end{equation}

To simplify the first term, we use the following formula \cite[Lemma 5.3]{chow2004ricci} relating the Gaussian curvature of conformal metrics:
\[
K_\eps = e^{-2\rho_\eps}\left(-\Delta_0\rho_\eps+K_0\right).
\]
Hence, $\dot K_0 = -2\dot\rho_0K_0-\Delta_0\dot\rho_0$, and thus
\begin{equation*}
\frac{d}{d \eps} \Big|_{\eps = 0} \sqrt{-K_{\eps}} = \frac{- \dot K_0}{2 \sqrt{-K_0}}
= -\dot \rho_0 \sqrt{-K_0} + \frac{\Delta_0 \dot{\rho_0}}{2 \sqrt{-K_0}}.
\end{equation*}
For the second term, we note that $dA_{\eps} = e^{2 \rho_\eps} dA_0$. This gives \begin{equation}\label{eq:area}
    \tfrac{d}{d \eps}|_{\eps=0} dA_{\eps} = 2 \dot \rho_0 d A_{0},
\end{equation}
    and combining everything completes the proof.
\end{proof}

Now we specialize to the normalized Ricci flow, i.e., we set $\dot{\rho_{0}} = -(K_{0} - \bar{K})$.
To prove our monotonicity result, we  use Lemma \ref{lemma: integral inequality} to show positivity of the second term in \eqref{eq:meancurv-conf}.

\begin{proof}[Proof of Theorem \ref{theo:meancurv}]
    Letting $\dot \rho_0 =-( K_0 - \bar{K})$ in Proposition \ref{meancurv-conf} and setting $F = \sqrt{-K_0} > 0$ gives
\begin{align*}
    \dot \kappa_0 &= {-}\int_M \frac{\Delta_0 K_0}{2 \sqrt{-K_0}} dA_0 {-} \int_M \sqrt{-K_0} (K_0 - \bar K) dA_0 = \int_M \frac{\Delta_0 F^2}{2F} dA_0 + \int_M (F^3 - F \int_M F^2) \, dA_0.
    \end{align*}
   For the first term, using Stokes' theorem yields
    \[
    \int_M \frac{\Delta_0 F^2}{2F} dA_0 = \int_M  \frac{1}{2}\Delta_0 F \, dA_0 + \int_M \frac{\Vert \nabla_0 F \Vert^2}{F} = \int_M \frac{\Vert \nabla_0 F \Vert^2}{F}\geq 0, 
    \]
    which is positive whenever $F$ (and hence $K_0$) is nonconstant.
    For the second term, we use 
\[
\int_M F \left( \int_M F^2 dA_0 \right) \, dA_0 = \left( \int_M F dA_0 \right) \left( \int_M F^2 dA_0 \right) = \int_M F^2 \left( \int_M F dA_0 \right) dA_0,
\]
to obtain
\[
 \int_M \left( F^3 - F \int_M F^2 dA_0\right) \, dA_0 = \int_M F^2 \left( F -  \int_M F dA_0 \right) dA_0.
\]
 By Lemma \ref{lemma: integral inequality}, this term is positive for $F$ non-constant. Hence, we showed that $\dot \kappa_0 > 0$.
\end{proof}

\section{Monotonicity of the Liouville entropy}\label{sec:Liou}

In Section \ref{sub:conf}, we will show Theorem \ref{thm:conf-deriv}, which expresses the derivative of the Liouville entropy with respect to an arbitrary conformal perturbation in terms of $w^s$ (see Section \ref{sec:w^s}) and the \emph{half-orbit integrals} defined in \eqref{eq:defIf}. After proving some properties of these half-orbit integrals in Section \ref{sec:half}, we then specialize to the case of the normalized Ricci flow in Section \ref{sub:specRicci} to deduce Theorem \ref{theo:Liou} when the initial metric is $1/6$-pinched. 

\subsection{Differentiating the Liouville entropy under conformal change}\label{sub:conf}

We consider a smooth $1$-parameter family of conformal metrics as in \eqref{eq:conformal}.
We  differentiate the Liouville entropy with respect to this general conformal deformation starting from \eqref{eq:Liouvilledef}.  

As in the work of Ledrappier--Shu \cite{LedShu_conformal}, our computation relies on the identification of $S \tilde M$ with $\tilde M \times \partial \tilde M$ introduced in \eqref{eq:Pi}.
Recall that $\pi_g\colon S^g\tilde{M}\to\partial\tilde M$ denotes the forward projection along the geodesic flow to the boundary at infinity.
Since the metrics $g_{\eps}$ are all quasi-isometric to $g_0$ (via the identity map), 
and $\partial \tilde M$ is a quasi-isometry invariant (see, for instance, \cite[Theorem III.H.3.9]{BH}), we can identify all the unit tangent bundles $S^g \tilde M$
using the identification 
$$\Pi_g: S^g\tilde M \to \tilde M \times \partial \tilde M, \quad (x, v) \mapsto (x, \pi_g(x,v)).$$
Let $M_0 \subset \tilde M$ be a fundamental domain for the action of the fundamental group of $M$ on $\tilde M$.
From now on, we will identify $S^gM$ with  $\Pi_g^{-1}(M_0 \times \partial \tilde M)$. 
We will write $\Pi_{g,x}$ for the restriction of $\Pi_g$ to the fiber $S_x^g\tilde M$. For $\xi \in \partial \tilde M$, recall from Section \ref{buse} that
$X^\xi_g\in C^{\infty}(\tilde M,T\tilde M)$
is a vector field such that $\pi_g(x,X^\xi_g(x))=\xi.$ 
Using \eqref{eq:Liouvilledef} and \eqref{eq:2ndff}, we see that the {Liouville entropy} is given by
\begin{equation}\label{eq:hDivX}
  h_{\mathrm{Liou}}(g) = - \int_{\tilde M \times \partial \tilde M} {\rm Div}_g(X^{\xi}_g)(x) \, (\Pi_g)_*dm_g(x, \xi).  
\end{equation}

To prove Theorem \ref{thm:conf-deriv}, we set $g = g_{\eps}$ for $g_{\eps}$ as in \eqref{eq:conformal} in the above formula and differentiate with respect to $\eps$. We use the subscript $\eps$ to denote objects defined from the metric $g_\eps$.

The fact that the Liouville entropy depends differentiably on the metric is non-trivial; this is due to Knieper and Weiss \cite{KniWei}  
for negatively curved surfaces, and to Contreras \cite{Con} for general negatively curved manifolds; see also \cite[(B1)]{Fla} for a more explicit formula. 
We will use a slightly different approach from \cite{KniWei} to compute the derivative by starting from \eqref{eq:hDivX} (the difference being that we integrate  ${\rm Div}(X)$ instead of the Riccati solution $w^s$). Formally differentiating \eqref{eq:hDivX} yields
\begin{equation}
\label{eq:D^1}
\begin{split}
    \left. \frac{d}{d \eps}\right|_{\eps = 0}& h_{\mathrm{Liou}}(\eps) = 
    \left.\frac{d}{d \eps}\right|_{\eps = 0} \left(
     -\int_{M_0 \times \partial \tilde M} {\rm Div}_{\eps}(X^{\xi})(x) \, (\Pi_{0})_*dm_0(x, \xi)\right.
     \\&\left.-\int_{M_0 \times \partial \tilde M} {\rm Div}(X_{\eps}^{\xi})(x) \, (\Pi_{0})_*dm_0(x, \xi)
    -\int_{M_0 \times \partial \tilde M} 
    {\rm Div}(X^{\xi})(x) \, (\Pi_{\eps})_*dm_\eps(x, \xi)
    \right).
    \end{split}
\end{equation}
We justify that the above formula makes sense by treating each term individually:
\begin{itemize}
    \item  The variation of the divergence can be computed using \cite{Bes}, see Lemma \ref{lem:divvar}.
 \item We note that the geodesic spray $X_{\eps}^\xi$ is differentiable when the metric varies, as shown by Ledrappier--Shu \cite[Theorem 3.11]{LedShu_general} (building on the work of Fathi--Flaminio \cite{FaFl}, and in turn on \cite[Theorem A.1]{dLMM}).
To simplify the second term in \eqref{eq:D^1}, we crucially use a formula of Ledrappier--Shu for the derivative $\partial_{\eps}|_{\eps=0} X_{\eps}^\xi$ (Proposition \ref{propDaniel}). 
 \item  The variation of the pushforward Liouville measure $(\Pi_{{\eps}})_* dm_{\eps}$ is computed in Lemma~\ref{lemm:VarLiou}. 
 \end{itemize}
 
\subsubsection{Computing the variation of the divergence}
To prove Theorem \ref{thm:conf-deriv}, we start by showing that the first term in (\ref{eq:D^1}) vanishes.
\begin{lemma}\label{lem:divvar}
With the notation introduced above, we have:
    \[  -\int_{M_0 \times \partial \tilde M} \left.\frac{d}{d\eps}\right|_{\eps=0}{\rm Div}_{\eps}(X^{\xi})(x) \, (\Pi_{0})_*dm_0(x, \xi) = 0.\]
\end{lemma}

\begin{proof}
    To compute the variation of $\mathrm{Div}_{\eps}$ with respect to $\eps$, we will use \cite[Theorem 1.174]{Bes}, which computes the variation of the Levi-Civita connection $\nabla^\eps$ associated to $g_\eps$. More precisely, for any vector fields $X,Y,Z\in C^{\infty}(M;TM)$, we have
\begin{equation}
\label{eq:Besse}g_0(\partial_{\eps}|_{\eps=0}\nabla^{\eps}X(Y),Z)=\frac 12 \big( \nabla_X \dot{g}_0(Y,Z)+\nabla_Y \dot{g}_0(X,Z)-\nabla_Z \dot{g}_0(X,Y)\big).  
\end{equation}
In particular, choosing a local orthonormal frame $(e_i)_{i=1,2}$, we obtain
\begin{align*} \partial_{\eps}|_{\eps=0}\mathrm{Div}_{\eps}(X^{\xi})=&\mathrm{tr}(\partial_{\eps}|_{\eps=0}\nabla^{\eps}X^{\xi})=\sum_{i=1}^2g_0(\partial_{\eps}|_{\eps=0}\nabla^{\eps}X^{\xi}(e_i),e_i) 
\\&=\frac 12 \sum_{i=1}^2\big( \nabla_{X^{\xi}} \dot{g}_0(e_i,e_i)+\nabla_{e_i} \dot{g}_0(X^{\xi},e_i)-\nabla_{e_i}\dot{g}_0(X^{\xi},e_i)\big)
\\&=\frac 1 2\mathrm{tr}(\nabla_{X^\xi} \dot{g_0})=\mathrm{tr}(X^{\xi}(\dot \rho_0)g_0+\dot \rho_0\nabla_{X^{\xi}}{g_0})=2X^{\xi}(\dot \rho_0).
\end{align*}
In the last line, we used $\dot g_0 = 2 \dot \rho_0g_0$ (see \eqref{eq:conformal}), followed by the Leibniz rule, followed by the fact that $\nabla_{X^{\xi}}{g_0}=0$.
In particular, we see that $\partial_{\eps}|_{\eps=0}\mathrm{Div}_{\eps}(X^{\xi})$ is a co-boundary. Indeed, for $X$ as in \eqref{eq:X}, one has $X^{\xi}(\dot \rho_0)(x)=X(\dot \rho_0)(\Pi_{0}^{-1}(x,\xi))$. Since the Liouville measure is invariant with respect to the geodesic flow by \eqref{eq:Adjoints}, this shows that the integral in the statement of the lemma vanishes. \end{proof}

\subsubsection{Computing the variation of the geodesic spray.}

To state the next proposition, we introduce the following \emph{half-orbit integrals.}
For any $f \in C^{1+\alpha}(SM)$, we define a new function $I_f$ on $SM$
by the following integral:
\begin{equation}\label{eq:defIf}
    I_f(v) = \int_0^{+\infty} \frac{j^s(\varphi_t v)}{j^s(v)} e^s(f)(\varphi_t v) \, dt,
\end{equation}
where $j^s(\varphi_t v)$ is a stable Jacobi function along the geodesic determined by $v$ (see Section \ref{buse}), and  $e^s$ is the stable vector field defined in \eqref{eq:e^s}.
Although $j^s$ is only defined up to a scalar factor, the ratio ${j^s(\varphi_t(v))}/{j^s(v)}$ is well defined along the geodesic and decreases exponentially fast when $t\to \infty$, see \eqref{eq:contract}. 
Since $e^s(f)$ is a continuous, and thus bounded, function on $SM$, the integral $I_f$ is convergent and defines a continuous function on $SM.$

Further properties of the above half-orbit integrals, which we will use to prove Theorem~\ref{theo:Liou}, are discussed in Section \ref{sec:half}.

\begin{proposition}\label{prop:second-term}
With the notations introduced above, we have
    \[\left.\frac{d}{d \eps}\right|_{\eps = 0} - \int_{M_0 \times \partial \tilde M} {\rm Div}(X_{\eps}^{\xi}) \, (\Pi_{0})_* dm_0(x, \xi) = -\int_{SM} I_{\dot \rho_0} V(w^s) dm_0 + \int_{SM} \dot \rho_0 w^s \, dm_0.
    \]
\end{proposition}

To prove this, we first use the work of  Ledrappier and Shu to compute the variation of $X_{\eps}^{\xi}$. 
We note that by construction, for each $x$, we have that $\eps \mapsto X_{\eps}^{\xi}(x)$ is a curve in $T_x \tilde M$. Below, we will view its tangent vector $\dot X^{\xi}(x)$ as an element of $T_x \tilde M \cong T_{(x,\xi)}(T_x \tilde M)$.

\begin{proposition}[\cite{LedShu_conformal}]\label{propDaniel}
Fix $v \in S M$ and let $(x, \xi):=\Pi_{0}(v)\in M_0 \times \partial \tilde M$. 
Then $\eps \mapsto X_\eps^\xi$ is differentiable at $\eps=0$ with derivative given by
\begin{equation}
    \label{eq:LDS}
  \left. \frac{d}{d \eps}\right|_{\eps=0} X_{\eps}^{\xi}=-\dot \rho_0 X^{\xi}+Y^{\xi},  
\end{equation}
where $Y^{\xi}$ is a vector field on $\tilde M$ perpendicular to $X^{\xi}$ and given by 
\begin{equation}
\label{eq:Y1}
Y^{\xi}(x):=-I_{\dot \rho_0}(v)JX^\xi(x),
\end{equation}
where $v=X^{\xi}(x)$, $J$ is the complex structure, and $I_{\dot \rho_0}$ is defined in
\eqref{eq:defIf}.\end{proposition}

\begin{proof} 
 For completeness, we give the proof of \cite[Proposition 4.5]{LedShu_conformal} specialized to the case of surfaces. We note that $\dot{X}^{\xi}$ can be naturally split into two terms as follows:
\begin{equation}\label{dotX}
\dot X^{\xi}_0 = \lim_{\eps \to 0} \frac{1}{\eps} \left( X^{\xi}_{\eps}(x) - \frac{X^{\xi}_{\eps}(x)}{\Vert X^{\xi}_{\eps}(x) \Vert}_{g_0} \right) + \lim_{\eps \to 0} \frac{1}{\eps} \left( \frac{X^{\xi}_{\eps}(x)}{\Vert X^{\xi}_{\eps}(x) \Vert}_{g_0} - X^{\xi}_{0}(x) \right).
\end{equation}
The first term records the variation of the $g_0$-length of $X_{\eps}^{\xi}$ and is equal to $\frac{d}{d \eps}|_{\eps=0} \Vert X_{\eps}^{\xi} \Vert_{g_0} X_{0}^{\xi}$.  
Differentiating $\|X^{\xi}_\eps\|_{g_\eps}\equiv 1$  using \eqref{eq:conformal} yields
$$\frac{d}{d \eps}\Big|_{\eps=0} \Vert  X^{\xi}_\eps \Vert_{g_0}=-\frac{d}{d \eps}\Big|_{\eps=0} \Vert  X^{\xi} \Vert_{g_\eps}=-\frac 1 2\dot{g_0}(v,v)=-\dot \rho_0.$$

The second term records the change of direction of $X_{\eps}^{\xi}$. Since it is the tangent vector to a curve in $S_x \tilde M$, it is of the form $f(x, \xi) JX^{\xi}(x)$ for some $f$.   
To compute $f$, consider $(t, \eps) \mapsto \gamma_{\eps}(t)$ where for each fixed $\eps$, the curve $t \mapsto \gamma_{\eps}(t)$ is the $g_{\eps}$-unit speed geodesic such that $\dot{\gamma}_\eps(0) =\Pi_\eps^{-1}(x, \xi)$. 
Let $Q(t) = \partial_{\eps}|_{\eps = 0} \gamma_{\eps}(t)\in S_{\gamma_0(t)}M$
and $q(t) = g_0(Q(t), JX^{\xi}(\gamma_0(t)))$. Since $\frac{D}{dt}Q(0) = \frac{d}{d\eps}|_{\eps = 0} X_{\eps}^{\xi}(x)$,
we have $$f(x, \xi) = g_0\left(\frac{D}{dt} Q(0),  JX^{\xi}\right) = q'(0),$$ where we used that $J$ is parallel and that $X^\xi$ is tangent to a geodesic.  To compute $q'(0)$, we first note that since $\gamma_{\eps}$ is a $g_{\eps}$-geodesic, we have $\nabla^{\eps}_{\dot \gamma_{\eps}} \dot \gamma_{\eps} = 0$. Differentiating at $\eps = 0$ yields the Jacobi-type equation
\begin{equation}\label{eq:forced-jac}
  \frac{D^2}{dt^2}Q(t) + K_0(\gamma_0(t))Q(t) = - \partial_{\eps }|_{\eps=0} \nabla^{\eps}_{\dot \gamma_0} \dot \gamma_0.  
\end{equation}
(Compare with \cite[(4.1)]{LedShu_conformal}.)
To simplify the right-hand side, using \eqref{eq:Besse} and $\nabla_{X^\xi}g_0=0,$
\begin{align*}g_0(- \partial_{\eps }|_{\eps=0} \nabla^{\eps}_{\dot \gamma_0} \dot \gamma_0,Z)= &-\frac 12 (2\nabla_{X^\xi}\dot{g_0}(X^\xi,Z)-\nabla_Z\dot{g}_0(X^\xi,X^\xi))
\\&=-2\nabla_{X^\xi}(\dot{\rho}_0g_0)(X^\xi,Z)+Z(\dot{\rho}_0)g_0(\dot \gamma_0,\dot \gamma_0)
\\&=-2X(\dot{\rho}_0)g_0(\dot \gamma_0,Z)+Z(\dot{\rho}_0),
\end{align*}
for any vector field $Z$.  Now we take the inner product of \eqref{eq:forced-jac} with $J X^{\xi}$ and use again that $J$ is parallel and $X^\xi$ is tangent to a geodesic to obtain
\begin{equation}
  \label{eq:Wronskien}  
q''(t) + K_0(\gamma_0(t))q(t) = JX^{\xi}(\dot \rho_0) = e^s(\dot \rho_0). 
\end{equation}
For the last equality, we note that 
    for a function $F$ on $M_0\times \partial \tilde M$ such that $F\circ \Pi_0\in C^1(S^0M)$, we have 
    \begin{equation}\label{eq:JX-e}
       (JX^{\xi} F)(x, \xi) = e^s (F \circ \Pi_0) (\Pi_0^{-1}(x, \xi)). 
    \end{equation}
Indeed, consider a path $s \mapsto c(s)$ such that $c(0) = x$ and $c'(0) = J X^\xi(x)$. Then, as explained in Section \ref{sec:w^s}, we have 
$e^s (F \circ \Pi_0)(v)=\tfrac{d}{ds}|_{s=0}F(c(s),\xi)=(JX^{\xi} F)(x, \xi). $

Denote by $j(t):=j^s(\varphi_t v)/j^s(v)$ the normalized stable Jacobi function as in \eqref{eq:defIf}.
As in \cite[Proof of Proposition 4.4]{LedShu_conformal}, we consider a Wronskian:
\begin{align*}    
\frac{d}{dt} &(j'(t)q(t) - q'(t)j(t)) = j''(t)q(t) -q''(t) j(t) 
\\&= -q(t) K_0(\gamma_0(t))j(t) + j(t)\big( K_0(\gamma_0(t))q(t)-e^s(\dot \rho_0)(\dot\gamma(t))\big)= -j(t) e^s(\dot \rho_0)(\dot\gamma(t)),
\end{align*}
where we used \eqref{eq:Wronskien} together with the Jacobi equation $j''(t) = - K_0(\gamma_0(t)) j(t)$.
Integrating from $0$ to $\infty$ yields
\[
j'(\infty) q(\infty) - q'(\infty) j(\infty) - (j'(0)q(0) - q'(0) j(0)) = q'(0)  = -\int_0^{\infty} j(t) e^s(\dot \rho_0)(\dot\gamma(t))dt=-I_{\dot{\rho}_0}(v),
\]
where we used that $q(0)=0$ and $j(0)=1$, and that $j(\infty)=0$ and $j'(\infty)=0$ by the Anosov property \eqref{eq:contract} together with the boundedness of $q$ and $q'$. 
This completes the proof.
\end{proof}

In order to prove Proposition \ref{prop:second-term}, we proceed to prove the following.
\begin{proposition}
\label{prop:V}Let $Y^{\xi}$ be the vector field in the statement of Proposition \ref{propDaniel}. With the notations introduced above, we have
  $$\int_{M_0\times \partial \tilde M}\mathrm{Div}(Y^{\xi})(x) \, (\Pi_{0})_*dm_0(x, \xi)= \int_{SM} I_{\dot{\rho}_0}(v) V(w^s)(v) \, dm_0(v).$$
\end{proposition}
 We will need the following lemma.
\begin{lemma}\label{lem:Div0} Fix $\xi \in \partial \tilde M$. Then
    ${\rm Div} (JX^{\xi})(x) = 0$ for all $x\in \tilde M$.
\end{lemma}

\begin{proof}
    We compute the divergence using the orthonormal basis $(X^{\xi}, JX^{\xi})$. 
We have, 
\begin{align*}
    {\rm Div}(JX^{\xi}) &= g_0( \nabla_{X^\xi} JX^{\xi}, X^{\xi} ) + g_0( \nabla_{JX^\xi} JX^{\xi}, JX^{\xi}) \\
    &= g_0( J \nabla_{X^\xi} X^{\xi}, X^{\xi} ) + g_0( J\nabla_{JX^\xi} X^{\xi}, JX^{\xi} ) \\
    &= 0 + g_0( \nabla_{JX^\xi} X^{\xi}, X^{\xi} )= \frac{1}{2} JX^{\xi} g_0( X^{\xi}, X^{\xi} ) = 0,
\end{align*}
where we used the facts that $J$ is parallel, that $J$ is an isometry, as well as the fact that $X^\xi$ is tangent to a geodesic. 
\end{proof} 

\begin{proof}[Proof of Proposition \ref{prop:V}] Note that by \eqref{eq:Y1}, we have $Y^\xi(x) = -I_{\dot \rho_0}(\Pi_{0}^{-1}(x,\xi)) JX^{\xi}(x)$. Let $v=\Pi_0^{-1}(x,\xi)$.
    Using the product rule for the divergence together with Lemma \ref{lem:Div0} and \eqref{eq:JX-e}, we obtain 
    $${\rm Div}(Y^{\xi})(x) = - JX^{\xi}(I_{\dot \rho_0} \circ \Pi_0^{-1})(x)= -e^s(I_{\dot \rho_0})(v).$$ 
    Now, using \eqref{eq:e^s} and \eqref{eq:Adjoints}, we integrate by parts to obtain \begin{align*}
    \int_{M_0\times \partial \tilde M}&\mathrm{Div}(Y^{\xi})(x) \, (\Pi_{0})_*dm_0(x, \xi) = -\int_{SM} e^s(I_{\dot \rho_0})(v)dm_0(v)
    \\&= -\int_{SM} H(I_{\dot \rho_0})(v) +  w^s(v) V(I_{\dot \rho_0})(v)dm_0(v) =  \int_{SM} V(w^s)(v) I_{\dot \rho_0}(v)dm_0(v),
 \end{align*}
 which completes the proof. 
\end{proof}

\begin{proof}[Proof of Proposition \ref{prop:second-term}]
    By Proposition \ref{propDaniel}, we have
    \begin{align*}
        -\int_{M_0 \times \partial \tilde M}\left. \frac{d}{d \eps}\right|_{\eps=0} {\rm Div}(X_{\eps}^{\xi})&\, (\Pi_{0})_* dm_0(x, \xi) = \int_{M_0 \times \partial \tilde M}{\rm Div}(\dot \rho_0 X^{\xi}) - {\rm Div}(Y^{\xi})\, (\Pi_{0})_* dm_0(x, \xi) \\
        &= \int_{M_0 \times \partial \tilde M}\dot \rho_0 {\rm Div}(X^{\xi})+X^\xi(\dot{\rho}_0) - {\rm Div}(Y^{\xi}) \, (\Pi_{0})_* dm_0(x, \xi)
\\&= \int_{M_0 \times \partial \tilde M}\dot \rho_0 {\rm Div}(X^{\xi})- {\rm Div}(Y^{\xi}) \, (\Pi_{0})_* dm_0(x, \xi),
    \end{align*}
    where we used that $m_0$ is invariant under the $g_0$-geodesic flow, see \eqref{eq:Adjoints}.
 Using \eqref{eq:2ndff} for the first term and applying Proposition \ref{prop:V} for the second term finishes the proof. 
\end{proof}

\subsubsection{Computing the variation of the pushforward Liouville measure.}
Next we compute the variation of the pushforward Liouville measure $  (\Pi_{\eps})_*dm_{\eps}$. 
Recall from Section \ref{subsub:Liou-meas} that the probability measure $dm_{\eps}$ is given locally by the product of the normalized Riemannian area $dA_{\eps}$ on the base $M$ together with the normalized spherical Lebesgue measure (arclength) $d\mathrm{Vol}|_{S^{g_\eps}_x M}$ on the circular fibers $S_x^{g_\eps} M$. 
Thus, its pushforward $ (\Pi_{{\eps}})_*dm_{\eps}$ by the projection $\Pi_{\eps}$ to $M_0 \times \partial \tilde M$ has the local product decomposition $d\lambda^{\eps}_x(\xi)dA_{\eps}(x) $, 
where 
\begin{equation}
    \label{eq:visual}
    \lambda^{\eps}_x:=(\Pi_{\eps,x})_*d\mathrm{Vol}|_{S^{g_\eps}_x M},
\end{equation} is the \emph{visual measure}. Recall that the variation of the area form was computed in \eqref{eq:area}. 

Thus, it remains to compute the derivative of the visual measure:
\begin{lemma}\label{lemm:VarLiou}
    Let $f(x, \xi)$ be a  function on $M_0 \times \partial \tilde M$ such that $f\circ \Pi_{0}\in  C^{1}(S^{g_0}M)$. Then
    \[
    \left.\frac{d}{d \eps}\right|_{\eps = 0} \int_{M_0 \times \partial \tilde M} f(x, \xi)  d\lambda^{\eps}_x(\xi)dA_0(x) = \int_{SM} I_{\dot{\rho}_0}(v) V(f\circ \Pi_{0})(v) \, dm_0(v).
    \]
\end{lemma}
\begin{proof}
We first define the fiberwise rescaling:
\begin{align*}
\Psi_{\eps}: S^{g_{\eps}} M \to S^{g_0} M, \qquad
(x,v) \mapsto \left(x,\frac{v}{\Vert v \Vert_{g_0}}\right)=\left(x,e^{\rho_\eps(x)}v\right),
\end{align*} and note that we have $d\mathrm{Vol}|_{S^\eps_x M}=(\Psi_\eps)^*d\mathrm{Vol}|_{S^0_x M}$. In particular, \eqref{eq:visual} implies that 
$\lambda_x^0=(T_\eps)_*\lambda_x^\eps, $
where $T_\eps=\Pi_0 \circ \Psi_{\eps} \circ \Pi_{\eps}^{-1}: M_0 \times \partial \tilde M \rightarrow M_0 \times \partial \tilde M$, i.e., we have the following commutative diagram.
$$\begin{tikzcd}
 M_0\times\partial \tilde {M} \arrow[r, "T_\eps"] \arrow[d, "\Pi_\eps^{-1}"]
& M_0\times\partial \tilde {M} \\
 S^{g_\eps}M \arrow[r, "\Psi_\eps"]
& S^{g_0}M \arrow[u, "\Pi_0"]
\end{tikzcd}$$
Then, doing a change of variables, we obtain
 \begin{align*}
         \left. \frac{d}{d \eps}\right|_{\eps=0} \int_{M_0 \times \partial \tilde M} f(x, \xi)  d\lambda^{\eps}_x(\xi)dA_0(x) &= \left. \frac{d}{d \eps}\right|_{\eps=0} \int_{M_0 \times \partial \tilde M} f \circ T_{\eps}^{-1} (x, \xi)  d \lambda_x^0(\xi)dA_0(x) \\
         &= - \left. \frac{d}{d \eps}\right|_{\eps=0} \int_{M_0 \times \partial \tilde M} f \circ T_{\eps} (x, \xi)  d \lambda_x^0(\xi)dA_0(x),
 \end{align*}
 where the last equality follows from differentiating $f = f \circ T_{\eps} \circ T_{\eps}^{-1}$ and noting that $T_0$ is the identity map. 
Next, note that $\eps \mapsto c_{x,\xi}(\eps):=\Psi_{\eps} \circ \Pi_{\eps}^{-1}(x, \xi) = \Psi_{\eps} (X^{\xi}_{\eps}(x))$ is a $C^1$ curve in $S^0_x \tilde M$ which is tangent at $\eps=0$ to the vertical lift of $Y^{\xi}(x)$; see the second term in \eqref{dotX}. 

The above integrand is thus equal to 
\begin{align*}
 \left. \frac{d}{d \eps}\right|_{\eps=0} (f \circ T_{\eps})(x, \xi)
 =\left. \frac{d}{d \eps}\right|_{\eps=0} (f\circ \Pi_{0})(c_{x,\xi}(\eps))=-I_{\dot{\rho}_0}V(f\circ \Pi_{0}),   
\end{align*}
where we used \eqref{eq:V} to obtain that the vertical lift of $JX^\xi$ is $V.$ This completes the proof.
\end{proof}

\begin{proof}[Proof of Theorem \ref{thm:conf-deriv}]
    We start with \eqref{eq:D^1} and simplify the first, second, and third terms using Lemma \ref{lem:divvar}, Proposition \ref{prop:second-term}, and \eqref{eq:area} together with Lemma \ref{lemm:VarLiou}, respectively. 
\end{proof}

\subsection{Half-orbit integrals}
\label{sec:half}
In order to obtain the positivity of the derivative from Theorem \ref{thm:conf-deriv} for $1/6$-pinched metrics, we will record some properties of the \emph{half-orbit integrals} $I_f$ defined in \eqref{eq:defIf}. We start by noticing that these functions satisfy a differential equation.

\begin{proposition}\label{prop:I_f}
\label{propintegral}
Fix $f\in C^{1+\alpha}(SM)$. 
Let $I_f\in C^0(SM)$ be defined as in \eqref{eq:defIf}.
Then $I_f$ is the unique continuous function which is differentiable in the flow direction satisfying
\begin{equation}
    \label{eq:eqdiffJf}
    (X+w^s)I_f=-e^s(f).
\end{equation}
\end{proposition}

\begin{proof}
    To check \eqref{eq:eqdiffJf}, we apply the pullback of the flow to $I_f$. 
    We have
    \begin{align*}(\varphi_\theta)^*I_f(v)&=\int_{0}^{+\infty}\frac{j^s(v_{\tau+\theta})}{j^s(v_{\theta})}[e^sf](v_{\tau+\theta})d\tau=\frac{1}{j^s(v_\theta)}\int_{\theta}^{+\infty}{j^s(v_{\tau})}[e^sf](v_{\tau})d\tau.
    \end{align*}
    Differentiating at $\theta=0$ gives
    $$X(I_f)=-\frac{X j^s}{j^s}(v)I_f(v)-\left.\frac 1{j^s(v)}j^s(v_\theta)[e^sf](v_\theta)\right|_{\theta=0}. $$
    Using that $X j^s/j^s=w^s$ shows that \eqref{eq:eqdiffJf} holds.
    
    To prove uniqueness, suppose that $u\in C^0(SM)$ is differentiable in the flow direction and solves \eqref{eq:eqdiffJf}. 
    Then $h:=I_f-u\in C^0(SM)$ solves
    $(X+w^s)h=0 $. Multiplying both sides by $h$ and integrating against the Liouville measure gives, using \eqref{eq:Adjoints},
    $$0=\frac 12\int_{SM}X(h^2)hdm_0=-\int_{SM}w^s h^2dm_0. $$
 Since $w^sh^2\leq 0 $, is continuous and since $w^s<0$, this implies $h\equiv 0$.
\end{proof}
Using the Riccati equation \eqref{eq:Riccati}, the function $V(w^s)$ appearing in Theorem \ref{thm:conf-deriv} can be reinterpreted as a half-orbit integral.
\begin{lemma}\label{lem:Vw^s}
    Let $V$ be the vertical vector field in \eqref{eq:V}. We have $I_{w^s}=V(w^s).$
    \end{lemma}
\begin{proof}
Since $VK = 0$, applying $V$ to both sides of the Riccati equation \eqref{eq:Riccati} gives
    $VXw^s=-2w^sVw^s$.
    Next, we use \eqref{eq:comm} to get
    $XVw^s+Hw^s=-2w^sVw^s $,
    which is equivalent to
    $$({X}+w^s)Vw^s=-(H+w^sV)w^s=-e^s(w^s). $$
    Since $V(w^s)\in C^\alpha(SM)$, see Section \ref{sec:w^s}, we deduce Lemma \ref{lem:Vw^s} from Proposition \ref{propintegral}.
\end{proof}
Another direct consequence of Proposition \ref{propintegral} is the following.
\begin{lemma}
\label{lemm:Xf}
    Let $f\in C^{1+\alpha}(SM)$ such that $X(f)\in C^{1+\alpha}(SM)$, then $
        I_{X(f)}=-e^s(f).$
\end{lemma}
\begin{proof}
    Using \eqref{eq:comm}, \eqref{eq:e^s} and \eqref{eq:Riccati}, we first obtain $[X, e^s] = -w^s e^s$. In particular,
    $$-(X+w^s)e^s(f)=-e^s(X(f))+w^se^s(f)-w^se^s(w^s)=-e^s(X(f)), $$
    and we deduce Lemma \ref{lemm:Xf} from the uniqueness statement in Proposition \ref{propintegral}.
\end{proof}

\subsection{Specializing to the normalized Ricci flow}\label{sub:specRicci} In this last subsection, we show Theorem \ref{theo:Liou}. We have $\dot \rho_0 = -(K_0 - \bar{K})$ in Theorem \ref{thm:conf-deriv}, and we obtain
\begin{equation}\label{eq:deriv-app}
    \left.\frac{d}{d \eps}\right|_{\eps = 0} h_{\mathrm{Liou}}(\eps) = \int_{SM}(K_0 - \bar{K})w^sdm -2 \int_{SM} I_{w^s} I_{-K_0}dm,
\end{equation}
where we used Lemma \ref{lem:Vw^s}, the linearity of $f \mapsto I_f$, and the fact that $I_{\bar{K}}=0$.

First, the following lemma shows that the first term appearing in \eqref{eq:deriv-app} is non-negative.
\begin{lemma}
    \label{second}
    For any negatively curved surface $(M,g)$, we have
    $\int_{SM}(K_g-\bar K)w^s_gdm_g\geq 0, $
    with equality if and only if $g$ has constant curvature.
\end{lemma}
\begin{proof}
   Integrating both sides of the Riccati equation \eqref{eq:Riccati} gives
    $\bar K=-\int_{SM} (w^s_g)^2 dm_g. $
    Next, multiplying \eqref{eq:Riccati} by $w^s_g$, we get
    $$\frac{1}2X((w_g^s)^2)=-(w_g^s)^3-K_gw_g^s. $$
    Integrating the above equality, we have $\int_{SM}K_gw^s_gdm_g=-\int_{SM}(w^s_g)^3dm_g$. Thus, using also Lemma \ref{lemma: integral inequality}, we obtain
    $$ \int_{SM}(K_g-\bar K)w^s_gdm_g=\int (w_g^s)^2\left(-w_g^s-\int (-w_g^s)dm_g\right)dm_g\geq 0, $$
    with strict inequality if $g$ has non-constant curvature.
\end{proof}

Next, we show that the second term appearing in \eqref{eq:deriv-app} is also non-negative when the metric is $1/6$-pinched.
We will use the results obtained in Section \ref{sec:half}. 
In particular, we have
\begin{equation}
        \label{eq:Y2}
        I_{-K}=-e^s(w^s)+I_{(w^s)^2}.
    \end{equation}
Indeed, this follows from  the Riccati equation \eqref{eq:Riccati}, the linearity of $f\mapsto I_f$, and Lemma~\ref{lemm:Xf} applied to $f=w^s.$ 

Using the differential equations obtained in Proposition \ref{prop:I_f} for $I_{w^s}$ and $I_{(w^s)^2}$ and \eqref{eq:Y2}, we have the following. 
\begin{proposition}
\label{prop:appendix}For any negatively curved metric $g$, we have the identity
    \begin{equation*}
    \label{eq:1/6}
        -  \int_{SM} I_{-K} I_{w^s}dm = \int_{SM} \frac{K}{2(w^s)^3} (I_{(w^s)^2} - w^s I_{w^s})^2dm + \int_{SM} -w^s(I_{w^s})^2 (3 + \frac{K}{2(w^s)^2})dm. 
         \end{equation*}
\end{proposition}

\begin{proof}
Using Proposition \ref{propintegral} and \eqref{eq:Y2}, we compute, using the $X$-invariance of $dm_0,$
\begin{align*}
     - \int_{SM} I_{-K} I_{w^s}dm
    &= \int_{SM} e^s(w^s) I_{w^s} dm- \int_{SM} I_{w^s} I_{(w^s)^2} dm\\
    &= -\int_{SM} (X I_{w^s} + w^s I_{w^s}) I_{w^s} dm- \int_{SM} I_{w^s} I_{(w^s)^2} dm
    \\
    &= \int_{SM} -w^s (I_{w^s})^2 dm
    - \int_{SM} I_{w^s} I_{(w^s)^2}dm.
\end{align*}
To simplify the second term above, we use Proposition \ref{prop:I_f} with $f = (w^s)^2$. This gives 
\begin{align*}
    - \int_{SM}{I_{w^s}}& I_{(w^s)^2} dm= \int_{SM}\frac{I_{w^s}}{w^s} (XI_{(w^s)^2} + 2 w^s e^s(w^s))dm
    \\=&- \int_{SM}X(I_{w^s}/w^s) I_{(w^s)^2}dm + \int_{SM}2  e^s(w^s) I_{w^s}dm \\
    =& - \int_{SM}X(1/w^s) I_{w^s} I_{(w^s)^2}dm - \int_{SM}\frac{1}{w^s} X(I_{w^s}) I_{(w^s)^2}dm+ \int_{SM}- 2 w^s (I_{w^s})^2dm,
\end{align*}
where we integrated by parts for the first term and used the same simplification as above for the last one. 
Next, note that by the Riccati equation \eqref{eq:Riccati}, we have
\[ - \int_{SM}X(1/w^s) I_{w^s} I_{(w^s)^2}dm = - \int_{SM}I_{w^s} I_{(w^s)^2}dm - \int_{SM}\frac{K}{(w^s)^2} I_{w^s} I_{(w^s)^2}dm.\]
Next, Proposition \ref{prop:I_f} with $f = w^s$ gives 
\begin{align*}
 - \int_{SM}\frac{1}{w^s} X(I_{w^s}) I_{(w^s)^2} dm&=  \int_{SM}I_{w^s} I_{(w^s)^2}dm + \int_{SM}\frac{e^s(w^s)}{w^s} I_{(w^s)^2}dm. 
\end{align*}
In total, we have 
\begin{align*}
    - \int_{SM} I_{-K} I_{w^s}dm = -\int_{SM} 3 w^s (I_{w^s})^2dm + \int_{SM}\frac{e^s(w^s)}{w^s} I_{(w^s)^2}dm- \int_{SM}\frac{K}{(w^s)^2} I_{w^s} I_{(w^s)^2}dm. 
\end{align*}
To simplify the second term, we use Proposition \ref{prop:I_f} with $f = (w^s)^2$, which gives
\begin{align*}
    - 2 \int_{SM}&\frac{e^s(w^s)}{w^s} I_{(w^s)^2}dm = \int_{SM}\frac{(I_{(w^s)^2})^2}{w^s}dm  + \int_{SM}\frac{X\big((I_{(w^s)^2})^2\big)}{2} \frac{1}{(w^s)^2}dm \\
    &= \int_{SM}\frac{(I_{(w^s)^2})^2}{w^s}dm + \int_{SM}(I_{(w^s)^2})^2 \frac{X(w^s)}{(w^s)^3}dm \\
    &= \int_{SM}\frac{(I_{(w^s)^2})^2}{w^s} \big( 1 -  \frac{(w^s)^2 + K}{(w^s)^2} \big)dm=- \int_{SM}\frac{K}{(w^s)^3} (I_{(w^s)^2})^2dm,
\end{align*}
where we integrated by parts and used \eqref{eq:Riccati}.
Hence, completing the square gives
\begin{align*}
    - \int_{SM}& I_{-K} I_{w^s}dm = \int_{SM}- 3 w^s (I_{w^s})^2 dm+ \int_{SM}\frac{K}{2 (w^s)^3} ((I_{(w^s)^2})^2  - 2 w^s I_{w^s} I_{(w^s)^2})dm \\
    =& \int_{SM}- 3 w^s (I_{w^s})^2dm + \int_{SM}\frac{K}{2 (w^s)^3} (I_{(w^s)^2}  - w^s I_{w^s})^2dm - \int_{SM}\frac{K}{2 (w^s)^3} (w^s)^2 (I_{w^s})^2dm \\
    =& \int_{SM}-w^s (I_{w^s})^2 (3 + \frac{K}{2 (w^s)^2})dm + \int_{SM}\frac{K}{2 (w^s)^3} (I_{(w^s)^2}  - w^s I_{w^s})^2dm.
\end{align*}
This completes the proof.
\end{proof}
We can now conclude the proof of Theorem \ref{theo:Liou}, which follows from the following proposition.
\begin{proposition}
\label{1/6PinchProp}
    Suppose that the metric $g$ is $1/6$-pinched. Then if $(g_\eps)_{\eps \geq 0}$ denotes the normalized Ricci flow starting from $g_0=g$, we have
    $$\left.\frac{d}{d \eps} \right|_{\eps = 0}h_{\mathrm{Liou}} (\eps)\geq 0, $$
    with equality if and only if $g$ is hyperbolic.
\end{proposition}
\begin{proof}
    Since the metric is $1/6$-pinched, there exists $C>0$ such that $-C\leq K\leq -\frac{1}{6}C$. Thus, by \cite[Appendix B, Lemma 1]{KniWei}, we have
    $\frac{1}{6}C\leq (w^s)^2\leq C$. In particular, 
    $$3+\frac 12 \cdot\frac{K}{(w^s)^2}\geq 3-\frac{1}{2}\cdot\frac{6C}{C}=0. $$
 By Proposition \ref{prop:appendix}, we see that $-\int_{SM} I_{w^s} I_{-K}dm\geq 0.$
    To conclude, we use \eqref{eq:deriv-app} and Lemma \ref{second}.
\end{proof}
\begin{proof}[Proof of Theorem \ref{theo:Liou}]
Let $(M,g_0)$ be a $1/6$-pinched metric. By \cite[Theorem 3.3]{hamilton1988ricci}, the pinching is preserved along the normalized Ricci flow, i.e., for any $\eps\geq 0$, the metric $g_\eps$ is also $1/6$-pinched. As a result, Theorem \ref{theo:Liou} follows from Proposition \ref{1/6PinchProp}. 
\end{proof}

\bibliography{references}{}
\bibliographystyle{alpha}

\end{document}